\documentclass[11pt]{article}
\usepackage{amssymb, amsmath, amsthm, graphicx, tikz}
\usepackage[left=1in,top=1in,right=1in]{geometry}
\date{}

\theoremstyle{definition}
\newtheorem{theorem}{Theorem}[section]
\newtheorem{lemma}[theorem]{Lemma}

\newtheorem{observation}[theorem]{Observation}
\counterwithin{equation}{section}
\allowdisplaybreaks

\title{On asymptotic packing of convex geometric and ordered graphs}
\author{Jiaxi Nie\thanks{Max Planck Institute for Mathematics in the Sciences, Leipzig, 04103 Germany. Email: {\tt jiaxi.nie@mis.mpg.de}}\and Erlang Surya\thanks{Department of Mathematics, University of California at San Diego, La Jolla, CA, 92093 USA. Email: {\tt esurya@ucsd.edu}. Supported by NSF grant DMS-2225631.} \and Ji Zeng\thanks{Department of Mathematics, University of California at San Diego, La Jolla, CA, 92093 USA. Email: {\tt jzeng@ucsd.edu}. Supported by NSF grant DMS-1800746.}}

\begin{document}

\maketitle

\begin{abstract}
A convex geometric graph $G$ is said to be packable if there exist edge-disjoint copies of $G$ in the complete convex geometric graph $K_n$ covering all but $o(n^2)$ edges. We prove that every convex geometric graph with cyclic chromatic number at most $4$ is packable. With a similar definition of packability for ordered graphs, we prove that every ordered graph with interval chromatic number at most $3$ is packable. Arguments based on the average length of edges imply these results are best possible. We also identify a class of convex geometric graphs that are packable due to having many ``long'' edges.
\end{abstract}

\section{Introduction}\label{sec:intro}
For two graphs $G$ and $H$, a \emph{combinatorial $G$-packing} of $H$ is a collection of edge-disjoint subgraphs of $H$ that are isomorphic to $G$. In the study of graph packing, we typically seek the maximum cardinality of a combinatorial $G$-packing of $H$, denoted by $p(H,G)$. Clearly, $p(H,G)\le |E(H)|/|E(G)|$. When equality holds, we call the corresponding collection a \emph{perfect combinatorial $G$-packing} of $H$. Due to a well-known result of Kirkman~\cite{Kirkman1847}, $K_n$ has a perfect combinatorial $K_3$-packing if and only if $n\equiv 1$ or $3\mod 6$. More generally, Wilson~\cite{wilson1975decompositions} proved that, for $n$ large enough, $K_n$ has a perfect combinatorial $G$-packing if and only if $|E(G)|$ divides $\binom{n}{2}$ and the greatest common divisor of the vertex degrees of $G$ divides $n-1$. This result is also generalized to hypergraphs \cite{glock2016existence,keevash2014existence}.

Although a perfect combinatorial $G$-packing of the complete graph $K_n$ doesn't always exist, we can pack any $G$ into complete graphs ``almost perfectly''. Let us call a graph $G$ \emph{combinatorially packable} if
\[\lim_{n\rightarrow\infty}\frac{p(K_n,G)|E(G)|}{|E(K_n)|}=1.\]
In other words, $G$ is combinatorially packable if there exist combinatorial $G$-packings of $K_n$ covering all but $o(n^2)$ edges. Using a technique known as R\"odl Nibble (see Theorem~\ref{Pippenger_Spencer}), one can easily show that all fixed graphs are combinatorially packable. See~\cite{Yuster2007} for a survey on combinatorial packing problems.

In this paper, we consider the packability of convex geometric graphs. A \emph{convex geometric graph} (shortly \emph{cgg}) $G$ is a graph with a cyclic order $\prec$ on its vertex set $V(G)$. Equivalently, it can be represented as a graph drawn in the plane such that $V(G)$ consists of points in convex position, $E(G)$ consists of straight-line segments connecting the corresponding points, and $\prec$ is the clockwise ordering of $V(G)$. Two cgg's $G$ and $H$ are \emph{isomorphic} if there is a graph isomorphism $f$ between them preserving the cyclic order, that is, $u\prec v\prec w$ if and only if $f(u)\prec f(v)\prec f(w)$. Note that any two complete cgg's of the same size are isomorphic, we use $K_n$ to denote a complete cgg on $n$ vertices. For cgg's $G$ and $H$, a \emph{convex $G$-packing} (shortly \emph{$G$-packing}) of $H$ is a collection of edge-disjoint sub-cgg's of $H$ that are isomorphic to $G$. We call a cgg $G$ \emph{convex-packable} (shortly \emph{packable}) if there exist (convex) $G$-packings of $K_n$ covering all but $o(n^2)$ edges.

In contrast to combinatorial packability, there exist cgg's that are not packable, for example, the plane cycle $C_5$ of length $5$. To see why $C_5$ isn't packable, we can consider the average length of edges. Inside a cgg $G$, an \emph{interval} of $V(G)$ is a subset that is contiguous with respect to $\prec$, and the \emph{length} $l_G(e)$ of an edge $e$ is one less than the smallest cardinality of an interval containing $e$. One can check that the average length of all edges in $K_n$ is $(1+o(1))n/4$. On the other hand, the average length of the edges in a copy of $C_5$ in $K_n$ is at most $n/5$, so the average length of all edges covered by a $C_5$-packing of $K_n$ is at most $n/5$. Hence, a $C_5$-packing can never cover all but $o(n^2)$ edges of $K_n$.

Partly based on this average length argument, Cranston, Nie, Verstra\"ete, and Wesolek~\cite{cranston2021asymptotic} recently determined all packable plane Hamiltonian cgg's. See Figure~\ref{fig:plane_hamiltonian}.

\begin{figure}[ht]
    \centering
    \begin{tikzpicture}[thick, scale=.7]
    \tikzstyle{uStyle}=[shape = circle, minimum size = 6.0pt, inner sep = 0pt,
    outer sep = 0pt, draw, fill=white]
    \tikzstyle{lStyle}=[shape = rectangle, minimum size = 20.0pt, inner sep = 0pt,
outer sep = 2pt, draw=none, fill=none]
    \tikzset{every node/.style=uStyle}
    
    \begin{scope}[xshift=-1.7in, yshift=1.5in]
    \foreach \i in {1,2,3}
    \draw (120*\i+90:1.6cm) node (v\i) {};
    \draw (v1)--(v2)--(v3)--(v1);
    
    \end{scope}
    
    \begin{scope}[xshift=0.15in, yshift=1.7in, scale=1.2]
    \foreach \i in {1,2,3,4}
    \draw (90*\i+45:1.6cm) node (v\i) {};
    \draw (v1)--(v2)--(v3)--(v4)--(v1);
    
    
    \end{scope}
    
    \begin{scope}[xshift=2.0in, yshift=1.7in, rotate=-45, scale=1.2]
            
    \foreach \i in {1,...,4}
    \draw (90*\i:1.6cm) node (v\i) {};
    
    \foreach \i/\j in
    {1/2,2/3,4/1,2/4,3/4}
    \draw (v\i) edge (v\j);
    
    \end{scope}
    
    \begin{scope}[xshift=-1.7in, yshift=-.06in, scale=.96]
        
    \foreach \i in {1,...,5}
    \draw (18+72*\i:1.6cm) node(v\i) {};
    \foreach \i/\j in
    {1/2,2/3,4/5,5/1,1/3,1/4,3/4}
    \draw (v\i) edge (v\j);
    
    \end{scope}
    
    \begin{scope}[xshift=0.15in]
    \foreach \i in {1,...,6}
    \draw (60*\i:1.6cm) node (v\i) {};
    
    \foreach \i/\j in
    {4/3,3/2,2/1,1/5,6/5,1/6,4/2,2/5,4/5}
    \draw (v\i) edge (v\j);
    
    \end{scope}
    
    \begin{scope}[xshift=2.0in]
    \foreach \i in {1,...,6}
    \draw (60*\i:1.6cm) node (v\i) {};
    
    \foreach \i/\j in
    {4/3,3/2,2/1,2/6,6/5,1/6,4/2,2/5,4/5}
    \draw (v\i) edge (v\j);
    
    \end{scope}
    
    \end{tikzpicture}
    \caption{All packable plane Hamiltonian cgg's.}
    \label{fig:plane_hamiltonian}
\end{figure}
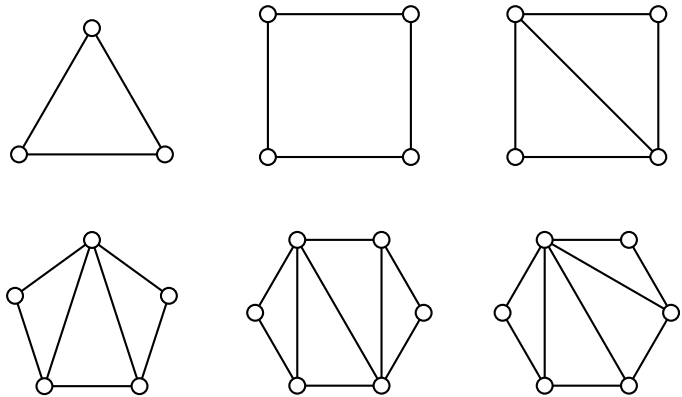

Because the average length of edges in $K_n$ is roughly $n/4$, it's natural to predict that a cgg is packable if it is ``cyclically 4-partite''. Our main result confirms this prediction. The \emph{cyclic chromatic number} $\chi_c(G)$ of a cgg $G$ is the minimum number of intervals that $V(G)$ can be partitioned into such that every edge in $E(G)$ has its endpoints in different parts.
\begin{theorem}\label{main}
Let $G$ be a cgg. If $\chi_c(G)\leq 4$, then $G$ is packable.
\end{theorem}
Theorem~\ref{main} is best possible since (plane) $C_5$ has cyclic chromatic number $5$ and is not packable.
To prove Theorem~\ref{main}, we prove a sufficient condition for the packability of a cgg (see Lemma~\ref{frac_pack_lemma}) which is satisfied by all cgg $G$ with $\chi_c(G)\le 4$. The sufficient condition is quite general, and we also use it to show that some cggs with higher cyclic chromatic numbers can still be packable if they have many ``long'' edges. Let $\mathcal{P}=\{V_1,\dots,V_k\}$ be an interval partition of $V(G)$ with each edge in $E(G)$ having its endpoints in different parts. Since every $V_i$ is an interval, we can define a cyclic order on $\mathcal{P}$ by setting $V_{i_1}\prec V_{i_2}\prec V_{i_3}$ if $v_{i_1}\prec v_{i_2}\prec v_{i_3}$ for any $v_{i_j}\in V_{i_j}$. For each pair $\{V_i,V_j\}$, we use $l_{\mathcal{P}}(\{V_i,V_j\})$ to denote the length of this pair in $\mathcal{P}$, and $E_G(V_i,V_j)$ to denote the set of edges between this pair.
\begin{theorem}\label{higher}
For every integer $k>2$ there is an absolute constant $C_k$ with the following property. Suppose $\mathcal{P}=\{V_1,\dots, V_{2k}\}$ is an interval partition of a cgg $G$ with each edge having its endpoints in different parts, and $E_l$ is the union of all $E_G(V_i,V_j)$ with $l_{\mathcal{P}}(\{V_i,V_j\})=l$ for $l\in [k]$. If $E_1\neq\emptyset$ and $|E_{k}|\geq C_k\cdot \sum_{l<k}|E_l|$, then $G$ is packable.
\end{theorem}

We also prove an analogue of Theorem~\ref{main} for ordered graphs. An ordered graph is a graph $G$ with a linear order $<$ on $V(G)$. An isomorphism between ordered graphs is a graph isomorphism preserving the linear order. Packings, intervals, and lengths are defined similarly for ordered graphs. We say an ordered graph $G$ is \emph{linearly packable} (shortly \emph{packable}) if there exist $G$-packings of the complete ordered graph $K_n$ that covers all but $o(n^2)$ edges. The \emph{interval chromatic number} $\chi_<(G)$ of an ordered graph $G$ is the minimum number of intervals that $V(G)$ can be partitioned into such that every edge in $E(G)$ has its endpoints in different parts.
\begin{theorem}\label{ordered}
Let $G$ be an ordered graph. If $\chi_<(G)\leq 3$, then $G$ is packable.
\end{theorem}
Theorem~\ref{ordered} is also best possible: the ordered graph $P_3$, whose vertices are $1<2<3<4$ and edges are $\{1,2\},\{2,3\},\{3,4\}$, is not packable. We give a heuristic explanation of this fact. The average length of all edges in the complete ordered graph $K_n$ is $(1+o(1))n/3$. Although the average length of the edges in a copy of $P_3$ in $K_n$ can attain $(1-o(1))n/3$, it only happens when the head and tail of this copy are very close to the head and tail of $V(K_n)$ with respect to $<$. But there are not enough such $P_3$-copies to cover most edges of $K_n$.

To end this introduction, we remark that research have been done for packing problems of (not necessarily convex) geometric graphs~\cite{AHKVLPSW2017,BK1979,biniaz2020packing,BHRW2006,obenaus2021edge,TCAK2019}. The extremal problems about cgg's and ordered graphs are also extensively studied, with many results particularly related to cyclic chromatic numbers and/or interval chromatic numbers~\cite{axenovich2018chromatic,BKV2003,furedi2020ordered,gyHori2018turan,pach2006forbidden}.

The rest of this paper is organized as follows. In Section~\ref{sec:frac_pack}, we prove Lemma~\ref{frac_pack_lemma} which translates the packing problems of cgg's into fractional packing problems of weighted cgg's. It will be useful to view the fractional packing problems as linear programming problems. In Section~\ref{sec:main} and Section~\ref{sec:higher}, we establish Theorem~\ref{main} and Theorem~\ref{higher}, by proving the feasibility of their corresponding linear programs through Farkas' lemma. In Section~\ref{sec:ordered}, we prove Theorem~\ref{ordered} directly using R\"odl Nibble. Section~\ref{sec:remark} lists some final remarks. We systemically omit floors and ceilings whenever they are not crucial for the sake of clarity in our presentation.

\section{Fractional packing of weighted convex geometric graphs}\label{sec:frac_pack}
In this section, we prove an auxiliary lemma which relates the packability of a given cgg to a fractional packing problem of an associated weighted cgg. We introduce some notions before we state and prove this lemma.

A \emph{weighted cgg} is a cgg $G$ together with a function $\omega_G: E(G)\to{\mathbb{R}_{\geq 0}}$. A weighted sub-cgg $G$ of $H$ is a weighted cgg whose underlying cgg structure is a sub-cgg of $H$. An unweighted cgg $H$ can be seen as a weighted cgg with each edge having weight $1$. An isomorphism $f:G\to H$ between two weighted cgg's is a cgg-isomorphism such that $\omega_G(e)=\omega_H(f(e))$ for all $e\in E(G)$. We use $\mathcal{I}(H,G)$ to denote the set of all weighted sub-cgg's of $H$ that are isomorphic to $G$. A \emph{perfect fractional $G$-packing} of $H$ is a function $\phi:\mathcal{I}(H,G)\to \mathbb{R}_{\geq 0}$ such that
\begin{align}\label{eq:perfect_fractional_packing_def}
    \sum_{S\in \mathcal{I}(H,G)}\phi(S)\omega_{S}(e)=\omega_H(e)\text{ for all $e\in E(H)$}.
\end{align} Here, we take the convention that $\omega_{S}(e)=0$ if $e\not\in E(S)$.

For an unweighted cgg $G$ and an interval partition $\mathcal{P}=\{V_1,\dots, V_k\}$ of $V(G)$ with each edge having its endpoints in different parts, we described in Section~\ref{sec:intro} that $\mathcal{P}$ has an induced cyclic order. A \emph{weighted representation} of $G$ (constructed from $\mathcal{P}$) is a complete weighted cgg $W$ with vertex set $\mathcal{P}$ and edge weight given by
\[ \omega_W(\{V_i,V_j\})=|E_G(V_i,V_j)|\text{  for every $i\neq j\in[k]$}.\]
See Figure~\ref{fig:weighted_and_blowup} for an example of weighted representations.

Now we state the main result of this section.
\begin{lemma}\label{frac_pack_lemma}
Let $G$ and $H$ be two cgg's. Suppose $\chi_c(G)\geq 3$ and $W$ is a weighted representation of $G$. If $H$ is packable and has a perfect fractional $W$-packing, then $G$ is packable.
\end{lemma}

To prove this lemma, we need the following theorem, known as R\"odl Nibble~\cite{Rodl1985}. The version stated here is due to Pippenger and Spencer~\cite{PS1989}, see also~\cite[Section~4.7]{AS2004}.
\begin{theorem}[R\"odl Nibble]\label{Pippenger_Spencer}
For every integer $r\ge 2$ and real numbers $c\ge1$ and $\epsilon>0$, there
exist $\gamma=\gamma(r,c,\epsilon)$, $n_0=n_0(r,c,\epsilon)$ and
$d_0=d_0(r,c,\epsilon)$ such that for all integers $n\ge n_0$ and $D\ge d_0$, the following holds. If $\mathcal{H}=(\mathcal{V},\mathcal{E})$ is an $n$-vertex $r$-uniform hypergraph satisfying
\begin{enumerate}
    \item For all $x\in V$ but at most $\gamma n$ of them, its degree $d(x)=(1\pm\gamma)D$.
    \item For all $x\in V$, its degree $d(x)<c D$.
    \item For any two distinct $x_1,x_2\in V$, their codegree $d(x_1,x_2)<\gamma D$.
\end{enumerate}
Then $\mathcal{H}$ contains a matching of at least $(1-\epsilon)n/r$ hyperedges.
\end{theorem}

Another ingredient in our proof is the notion of blowups of cgg's. Given a cgg $H$ and a positive integer $n$, the \emph{$n$-regular blowup} $H[n]$ of $H$ is a cgg constructed as follows. The vertex set of $H[n]$ consists of $|V(H)|$ intervals, denoted as $\{I_v\}_{v\in V(H)}$, each of cardinality $n$, and $I_{v_1}\prec I_{v_2}\prec I_{v_3}$ in $H[n]$ whenever $v_1\prec v_2\prec v_3$ in $H$; The edge set of $H[n]$ consists of all pairs between $I_{u}$ and $I_{v}$ whenever $\{u,v\}\in E(H)$. See Figure~\ref{fig:weighted_and_blowup} for an example of blowups.

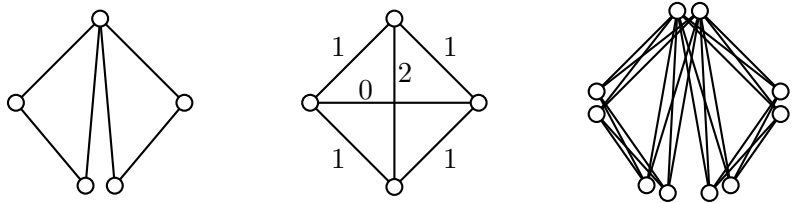
\begin{figure}[ht]
    \centering
    \begin{tikzpicture}[thick, scale=.7]
    \tikzstyle{uStyle}=[shape = circle, minimum size = 6.0pt, inner sep = 0pt,
    outer sep = 0pt, draw, fill=white]
    \tikzstyle{lStyle}=[shape = rectangle, minimum size = 20.0pt, inner sep = 0pt,
outer sep = 2pt, draw=none, fill=none]
    \tikzset{every node/.style=uStyle}
    
    \begin{scope}[xshift=-2.2in]
        
    \foreach \i in {1,2,3}
    \draw (90*\i-90:1.6cm) node(v\i) {};
    \draw (260:1.6cm) node(v4){};
    \draw (280:1.6cm) node(v5){};
    \foreach \i/\j in
    {1/2,2/3,3/4,4/2,2/5,1/5}
    \draw (v\i) edge (v\j);
    
    \end{scope}
    
    \begin{scope}[xshift=0in]
    \foreach \i in {1,2,3,4}
    \draw (90*\i:1.6cm) node (v\i) {};
    
    \foreach \i/\j in
    {1/2,1/3,1/4,2/3,2/4,3/4}
    \draw (v\i) edge (v\j);
    
    \foreach \i in{1,2,3,4}
    \draw (45+90*\i:1.5cm) node[lStyle]{$1$};
    \draw (70:0.6cm) node[lStyle]{$2$};
    \draw (155:0.6cm) node[lStyle]{$0$};
    \end{scope}
    
    \begin{scope}[xshift=2.2in,scale=1.1]
         
    \foreach \i in {1,2,3}
    \draw (90*\i-97:1.6cm) node(v\i1) {};
    \foreach \i in {1,2,3}
    \draw (90*\i-83:1.6cm) node(v\i2) {};
    \draw (243:1.6cm) node(v41){};
    \draw (283:1.6cm) node(v51){};
    \draw (257:1.6cm) node(v42){};
    \draw (297:1.6cm) node(v52){};
    \foreach \i/\j in
    {1/2,2/3,3/4,4/2,2/5,1/5}
    \draw (v\i1) edge (v\j1);
    \foreach \i/\j in
    {1/2,2/3,3/4,4/2,2/5,1/5}
    \draw (v\i1) edge (v\j2);
    \foreach \i/\j in
    {1/2,2/3,3/4,4/2,2/5,1/5}
    \draw (v\i2) edge (v\j1);
    \foreach \i/\j in
    {1/2,2/3,3/4,4/2,2/5,1/5}
    \draw (v\i2) edge (v\j2);
    \end{scope}
    
    \end{tikzpicture}
    \caption{A cgg $H$; A weighted representation of $H$; The blowup $H[2]$.}
    \label{fig:weighted_and_blowup}
\end{figure}

\begin{proof}[Proof of Lemma~\ref{frac_pack_lemma}]
We claim that there exist $G$-packings of the blowup $H[n]$ covering all but $o(n^2)$ edges. We first prove $G$ is packable given this claim. Take an arbitrary positive integer $n$. Since $H$ is packable, there exists an $H$-packing $\mathcal{C}$ of $K_n$ that covers all but $o(n^2)$ edges. Observe that $\mathcal{C}$ induces an $H[n]$-packing $\mathcal{C}'$ of the blowup $K_n[n]$ as follows: if $H'\in \mathcal{C}$ is a sub-cgg of $K_n$, then $H'[n]$ is naturally a sub-cgg of $K_n[n]$ and we include it in $\mathcal{C}'$. For each copy of $H[n]$ in $\mathcal{C}'$, according to the claim, we can use a $G$-packing to cover all but $o(n^2)$ of its edges, and consequently obtain a $G$-packing $\mathcal{C}''$ of $K_n[n]$. Since $K_n[n]$ is a sub-cgg of $K_{n^2}$, we view $\mathcal{C}''$ as a $G$-packing of $K_{n^2}$, whose uncovered edges are in either one of the following categories:\begin{itemize}
    \item Edges that are in $K_{n^2}$ but not in $K_n[n]$. There are $O(n^3)$ many of them;
    \item Edges of $K_n[n]$ that are not covered by $\mathcal{C}'$. There are at most $o(n^2)\cdot n^2=o(n^4)$ of them (the $o(n^2)$ term corresponds to edges of $K_n$ not covered by $\mathcal{C}$);
    \item Edges of copies of $H[n]$ in $\mathcal{C}'$ that are not covered by $\mathcal{C}''$. There are at most $|\mathcal{C}'|\cdot o(n^2)=o(n^4)$ of them (note that $|\mathcal{C}'|<|E(K_n)|=O(n^2)$).
\end{itemize} Therefore, the $G$-packing $\mathcal{C}''$ covers all but $o(n^4)$ edges of $K_{n^2}$. So $G$ is packable by the arbitrariness of $n$.

For the rest of this proof, we show the existence of $G$-packings of $H[n]$ covering all but $o(n^2)$ edges. Let $\mathcal{P}=\{V_1,\dots,V_k\}$ be the interval partition from which $W$ is constructed. Note that $k\geq\chi_c(G)\geq 3$. Denote each interval in the blowup $H[n]$ corresponding to a vertex $v\in V(H)$ as $I_v$. By assumption, there's a function $\phi:\mathcal{I}(H,W)\to \mathbb{R}_{\geq 0}$ satisfying \eqref{eq:perfect_fractional_packing_def} for $W$. We write
\[\Phi:=\max_{S\in \mathcal{I}(H,W)}\phi(S)\quad\text{and}\quad m:=\max\{|V_1|,\dots,|V_k|\}.\]

We construct a random hypergraph $\mathcal{H}$ whose vertex set $\mathcal{V}$ consists of all edges of $H[n]$ and edge set $\mathcal{E}$ consists of copies of $G$ coming from the following random experiment: For each sub-cgg $G'\subset H[n]$ satisfying\begin{itemize}
    \item there exists an isomorphism $f:G\to G'$ such that there are vertices $v_1,\dots,v_k\in V(H)$ with $f(V_i)\subset I_{v_i}$, and
    \item there exists $\delta\in [1,\log n]$ such that for all $i$, every consecutive pair of vertices in $V(G')\cap I_{v_i}$ has length $\delta$ in $H[n]$,
\end{itemize} its weighted representation $W'$ is naturally a weighted sub-cgg of $H$ (on vertices $v_1,\dots,v_k$). See Figure~\ref{fig:frac_pack_correspondence}. We include $E(G')$ into $\mathcal{E}$ as a hyperedge with probability $\phi(W')/\Phi$.

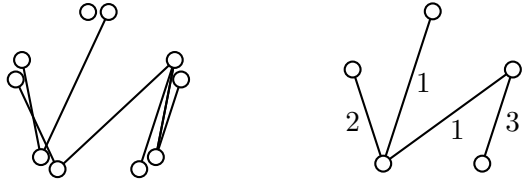
\begin{figure}[ht]
    \centering
    \begin{tikzpicture}[thick, scale=.7]
    \tikzstyle{uStyle}=[shape = circle, minimum size = 6.0pt, inner sep = 0pt,
    outer sep = 0pt, draw, fill=white]
    \tikzstyle{lStyle}=[shape = rectangle, minimum size = 20.0pt, inner sep = 0pt,
outer sep = 2pt, draw=none, fill=none]
    \tikzset{every node/.style=uStyle}
    
    \begin{scope}[xshift=-2.5in]
        
    \foreach \i in {1,2,3,4,5}
    \draw (72*\i+25:1.6cm) node(v\i1) {};
    \foreach \i in {1,2,3,4,5}
    \draw (72*\i+11:1.6cm) node(v\i2) {};

    \foreach \i/\j in
    {12/32, 32/22, 21/31, 31/51, 51/41, 51/41, 51/42, 41/52}
    \draw (v\i) edge (v\j);
    
    \end{scope}
    
    \begin{scope}
    \foreach \i in {1,2,3,4,5}
    \draw (72*\i+18:1.6cm) node (v\i) {};
    
    \foreach \i/\j in
    {1/3,2/3,3/5,5/4}
    \draw (v\i) edge (v\j);
    
    \draw (126:0.3cm) node[lStyle]{$1$};
    \draw (306:0.8cm) node[lStyle]{$1$};
    \draw (342:1.6cm) node[lStyle]{$3$};
    \draw (198:1.6cm) node[lStyle]{$2$};

    \end{scope}
    
    \end{tikzpicture}
    \caption{A sub-cgg $G'\subset H[2]$; Its weighted representation $W'\subset H$. (Edges of $H$ and $H[2]$ are not pictured here but in Figure~\ref{fig:weighted_and_blowup}.)}
    \label{fig:frac_pack_correspondence}
\end{figure}

Now we check that $\mathcal{H}$ satisfies the conditions of R\"odl Nibble with high probability.

Take any vertex $x\in \mathcal{V}$ and denote its degree in $\mathcal{H}$ as $D_x$. Write $x=\{u_1,u_2\}$ as an edge in $H[n]$. Consider $v_i\in V(H)$ such that $u_i\in I_{v_i}$ for $i=1,2$ and let $e:=\{v_1,v_2\}\in E(H)$. We say $x$ is \emph{good} if $u_1$ and $u_2$ are at least $m\log n$ vertices away from both boundary vertices of the interval $I_{v_i}$. If $x$ is good, then we can compute the expectation
\[\mathbb{E}(D_x)=(1-o(1))\frac{n^{k-2}\log n}{\Phi}.\]
Indeed, for each $W'\in \mathcal{I}(H,W)$ containing $e$, we count the number of $G'$ in $\mathcal{H}$ whose weighted representation is $W'$ and covering the edge $x\in E(H[n])$. There are $\log n$ choices for $\delta$, $\omega_{W'}(e)$ choices for $f^{-1}(x)$, and $(1-o(1))n^{k-2}$ choices for vertices in $V(G')\cap I_v$ with $v\in V(W')\setminus e$. Hence the computation follows from linearity of expectations, $H$ being unweighted, and \eqref{eq:perfect_fractional_packing_def}.

If $x$ is not good, there are fewer choices in the above process, so $\mathbb{E}(D_x)\leq (1-o(1))n^{k-2}\log n/\Phi$. Moreover, there are at most $\binom{|V(H)|}{2}(2m\log n)2n=O(n\log n)$ edges of $H[n]$ that are not good.

Take any two distinct vertices $x_1,x_2\in \mathcal{V}$ and denote their codegree in $\mathcal{H}$ by $D_{x_1x_2}$. Write $x_i=\{u_{i1},u_{i2}\}\in E(H[n])$ and let $v_{ij}\in V(H)$ such that $u_{ij}\in I_{v_{ij}}$ for $i,j=1,2$. Consider the two edges $e_i=\{v_{i1},v_{i2}\}$. If $e_1=e_2$, we can compute the expectation
\[\mathbb{E}(D_{x_1x_2})\leq (1-o(1))\frac{m^2n^{k-2}}{\Phi}.\]
Indeed, for each $W'\in \mathcal{I}(H,W)$ containing $e:=e_1=e_2$, we count the number of $G'$ in $\mathcal{H}$ whose weighted representation is $W'$ and covering the edges $x_1,x_2\in E(H[n])$. There are at most $\omega_{W'}(e)$ choices for both $f^{-1}(x_1)$ and $f^{-1}(x_2)$. Note that $\omega_{W'}(e)\leq m^2$. However, after those two choices, $\delta$ will be fixed. And there are still $(1-o(1))n^{k-2}$ choices for vertices in $V(G')\cap I_v$ with $v\in V(W')\setminus e$. Hence the computation follows again.

If $e_1\neq e_2$, we can check $\mathbb{E}(D_{e_1e_2})\leq (1-o(1))m^2n^{k-3}\log n/\Phi$ by similar counting and computation. Essentially, there are fewer choices for vertices in $V(G')\cap I_v$ with $v\in V(W')\setminus (e_1\cup e_2)$ after $\delta$ is fixed.

Finally, we check these degrees and codegrees are concentrated around their expectations. Observe that $D_x$ is a summation of independent random variables all taking values from $\{0,1\}$. By Chernoff bound (\cite[Corollary A.1.14]{AS2004}), for any $x\in \mathcal{V}$, we have
\[\Pr\left(|D_x-\mathbb{E}(D_x)|>\frac{1}{\sqrt{\log n}} \mathbb{E}(D_x)\right)<2\exp\left(\frac{-\mathbb{E}(D_x)}{3\log n}\right)\leq\exp(-\Omega(n)).\]
Here we used $k\geq 3$. On the other hand, there are at most $\binom{|V(H)|}{2}n^2$ many vertices in $\mathcal{V}$. Therefore, by union bound, the probability that every $x\in \mathcal{V}$ satisfies $|D_x-\mathbb{E}(D_x)|\leq \mathbb{E}(D_x)/\sqrt{\log n}$ converges to $1$ as $n$ tends to infinity. By a similar argument, the probability that every distinct pair $x_1,x_2\in \mathcal{V}$ satisfies $D_{x_1x_2}<2m^2n^{k-2}/\Phi$ converges to $1$ as $n$ tends to infinity.

To summarise, we have shown, when $n$ is large enough, $\mathcal{H}$ almost surely satisfies: All $x\in \mathcal{V}$ but $O(n\log n)$ of them have degree $d(x)=(1\pm o(1))n^{k-2}\log n/\Phi$; All $x\in \mathcal{V}$ have degree $d(x)<2n^{k-2}\log n/\Phi$; Any two distinct $x_1,x_2\in \mathcal{V}$ have codegree $d(x,y)<o(n^{k-2}\log n)$. Thus, there is a realization of $\mathcal{H}$ that satisfies the conditions of Theorem~\ref{Pippenger_Spencer}, meaning there is a matching which corresponds to $G$-packings of $H[n]$ covering all but $o(n^2)$ edges, as claimed.
\end{proof}

\section{Proof of Theorem~\ref{main}}\label{sec:main}
To prove Theorem~\ref{main}, we need the following lemma.
\begin{lemma}\label{K4_frac_pack}
Let $W$ be a complete weighted cgg on vertices $v_1\prec v_2\prec v_3\prec v_4$ with\begin{align*}
    &\omega_W(\{v_{1},v_2\})=\omega_W(\{v_{2},v_3\})=\omega_W(\{v_{3},v_4\})=\omega_W(\{v_{4},v_1\})>0,\\
    &\omega_W(\{v_{1},v_3\})=\omega_W(\{v_{2},v_4\})>0.
\end{align*} Then there exists some $K_m$ having a perfect fractional $W$-packing.
\end{lemma}

Before we establish this lemma, we first use it to prove Theorem~\ref{main}. For a complete cgg whose vertices are $v_1\prec v_2\prec\dots\prec v_m$ in cyclic order, its \emph{rotation} refers to an automorphism that maps $v_i$ to $v_{i+r}$ for all $i\in [m]$ and some fixed $r\in \mathbb{Z}$, where the indices are meant modulo $m$.

\begin{proof}[Proof of Theorem~\ref{main}]
Let $G$ be the given cgg. If $\chi_c(G)=2$, a result of Brass, K\'arolyi and Valtr~\cite{BKV2003} states that
\[\text{ex}_{c}(n,G)=\left(1-\frac{1}{\chi_c(G)-1}\right)\binom{n}{2}+o(n^2)=o(n^2),\]
where $\text{ex}_{c}(n,G)$ is the maximum number of edges in a cgg on $n$ vertices that doesn't contain $G$ as a sub-cgg. Clearly, for a $G$-packing of $K_n$ with maximal cardinality, the number of uncovered edges is less than $\text{ex}_{c}(n,G)$. So for this case we conclude $G$ is packable.

If $\chi_c(G)=3$, let $\mathcal{P}$ be an interval partition of $G$ into 3 parts with each edge having its endpoints in different parts, and $W$ the weighted representation of $G$ constructed from $\mathcal{P}$. Write the vertices of $W$ as $v_1\prec v_2\prec v_3$, and consider another weighted cgg $W'$ on these vertices with $\omega_{W'}(\{v_1,v_2\})$, $\omega_{W'}(\{v_2,v_3\})$, and $\omega_{W'}(\{v_3,v_1\})$ all equal to
\[\omega_W(\{v_1,v_2\})+\omega_W(\{v_2,v_3\})+\omega_W(\{v_3,v_1\}).\]
By considering the rotations of $W$, we can see that $W'$ has a perfect fractional $W$-packing. The unweighted triangle $K_3$ has a perfect fractional $W'$-packing just by scaling. So $K_3$ has a perfect fractional $W$-packing and because $K_3$ is packable (see Figure~\ref{fig:plane_hamiltonian}), using Lemma~\ref{frac_pack_lemma} we conclude the theorem.

If $\chi_c(G)=4$, let $W$ be a weighted representation of $G$ constructed from an interval partition with 4 parts. Write the vertices of $W$ as $v_1\prec v_2\prec v_3\prec v_4$ and consider another weighted cgg $W'$ on these vertices such that every edge with length $l$ in $W'$ has weight $w_l$, where
\begin{align*}
    &w_1:=\omega_W(\{v_1,v_2\})+\omega_W(\{v_2,v_3\})+\omega_W(\{v_3,v_4\})+\omega_W(\{v_4,v_1\}),\\
    &w_2:=2\omega_W(\{v_1,v_3\})+2\omega_W(\{v_2,v_4\}).
\end{align*}
Similarly, $W'$ has a perfect fractional $W$-packing by rotations. It's easy to check that $\chi_c(G)=4$ guarantees $w_1>0$. If $w_2=0$, then the plane cycle $C_4$ of length 4 has a perfect fractional $W'$-packing just by scaling. So $C_4$ has a perfect fractional $W$-packing and because $C_4$ is packable (see Figure~\ref{fig:plane_hamiltonian}), by Lemma~\ref{frac_pack_lemma} we are done. If $w_2>0$, by Lemma~\ref{K4_frac_pack}, a complete cgg $K_m$ has a perfect fractional $W'$-packing, hence $K_m$ also has a perfect fractional $W$-packing. Then we conclude Theorem~\ref{main} by Lemma~\ref{frac_pack_lemma} and the next observation.
\end{proof}
\begin{observation}\label{Km_packable}
Any complete (unweighted) cgg $K_m$ is packable.
\end{observation}
\begin{proof}
Since every abstract graph is combinatorially packable, there are combinatorial $K_m$-packings of $K_n$ that cover all but $o(n^2)$ edges. But we know that all complete cgg's of the same size are isomorphic to each other, so such combinatorial $K_m$-packings are also convex $K_m$-packings.
\end{proof}

The remainder of the section is dedicated to the proof of Lemma~\ref{K4_frac_pack}, which is based on the well-known Farkas' lemma~\cite{matousek2006understanding}.
\begin{lemma}[Farkas' Lemma]\label{farkas}
Let $M\in \mathbb{R}^{a\times b}$ and $z\in \mathbb{R}^a$. Then exactly one of the following two assertions is true:
\begin{enumerate}
    \item There exists $x\in \mathbb{R}^b$ such that $x\geq 0$ and $Mx=z$.
    \item There exists $y\in \mathbb{R}^a$ such that $z^Ty<0$ and $M^Ty\geq 0$.
\end{enumerate}
\end{lemma}

To approach Lemma~\ref{K4_frac_pack} using the Farkas' lemma, we define the following matrix: Two sub-cgg's $S_1,S_2\in \mathcal{I}(K_m,W)$ are \emph{rotationally equivalent} if there is a rotation of $K_m$ that transforms $S_1$ to $S_2$. For every $S\in \mathcal{I}(K_m,W)$, we denote its rotational equivalence class by $\overline{S}$. Let $M_{K_m,W}$ be the matrix whose rows are indexed by all possible edge-lengths in $K_m$, and columns indexed by rotational equivalence classes in $\mathcal{I}(K_m,W)$, and set\begin{equation}\label{eq:compressed_matrix_def}
    M_{K_m,W}(l,\overline{S})=\sum_{e\in E(S);l_{K_m}(e)=l}\omega_{S}(e).
\end{equation}

\begin{observation}\label{compressed_matrix}
If $m$ is odd and there exists a vector $x\geq 0$ such that $M_{K_m,W}x=\Vec{1}$, then $K_m$ has a perfect fractional $W$-packing.
\end{observation}
\begin{proof}
For every $S\in \mathcal{I}(K_m,W)$, let $\phi(S)$ be the entry of $x$ indexed by $\overline{S}$. Observe that for any two edges $e_1,e_2\in E(K_m)$ having the same length, there's a unique rotation that transforms $e_1$ to $e_2$. (This requires $m$ to be odd.) Also note that every edge of $K_m$ has weight 1. We can check that $M_{K_m,W}x=\Vec{1}$ is equivalent to \eqref{eq:perfect_fractional_packing_def}.
\end{proof}

\begin{proof}[Proof of Lemma~\ref{K4_frac_pack}] Write the weight of length-$l$ edges in $W$ as $w_l$ for $l=1,2$, and set
\[m=2m'+1\text{\quad with\quad}m'=\left\lceil\frac{12(w_1+w_2)^2}{w_1w_2}\right\rceil.\]
Note that $m'$ is the largest edge-length in $K_m$. For the rest of this proof, all the edge-lengths are defined in $K_m$.

By Observation~\ref{compressed_matrix}, it suffices to show there exists some vector $x\geq 0$ with $M_{K_m,W}x=\Vec{1}$. For the sake of contradiction, suppose there's no such $x$. By Lemma~\ref{farkas}, there exists $y\in \mathbb{R}^{m'}$ such that\begin{itemize}
    \item $\sum^{m'}_{i=1}y_{i}=(\Vec{1})^T\cdot y<0$; and
    \item For any $S\in \mathcal{I}(K_m,W)$, whose $w_1$-weighted edges have lengths $i_1,i_2,i_3,i_4$ and $w_2$-weighted edges have lengths $i_5,i_6$, we have
    \[ y(\overline{S}):= w_1(y_{i_1}+y_{i_2}+y_{i_3}+y_{i_4})+w_2(y_{i_5}+y_{i_6})=(M^T_{K_m,W}\cdot y)_{\overline{S}}\geq 0.\]
\end{itemize} Let $I:=\{i\in [m']; y_i<0\}$, $N=\sum_{i\in I} y_i$ and $P=\sum_{i\not\in I} y_i$. Then we have $N+P<0$.

Let $i_{max}$ be the maximum element in $I$. For each $i\in I\setminus i_{max}$, choose $S_i\in \mathcal{I}(K_m,W)$ such that its $w_1$-weighted edges have lengths $i,i_{max}-i,i,\min\{i_{max}+i, m-i_{max}-i\}$ following the cyclic order. See Figure~\ref{fig:trapezoid}. Note that $\overline{S_i}$ is uniquely determined in this way. Consider all the edges from all chosen $S_i$: for each $i\in I\setminus i_{max}$, there are at least two $w_1$-weighted edges having length $i$; for each $j\in [m']\setminus I$, there are at most two $w_1$-weighted edges having length $j$ (notice that $j$ is not equal to both $i_{max}+i_1$ and $i_{max}-i_2$ for $i_1,i_2\in I$); there are $2(|I|-1)$ many $w_2$-weighted edges having length $i_{max}$. So we have\begin{align*}
    \sum_{i\in I\setminus i_{max}} y(\overline{S_i})&\leq 2w_1(N-y_{i_{max}})+2w_1P+2(|I|-1)w_2y_{i_{max}}\\
    &<2((|I|-1)w_2-w_1)y_{i_{max}},
\end{align*}where the last inequality holds since $N+P<0$.

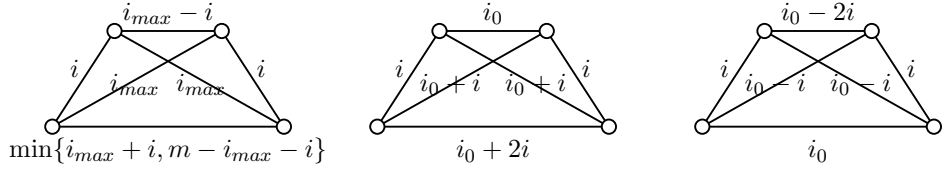
\begin{figure}[ht]
    \centering
    \scalebox{0.9}{\begin{tikzpicture}[thick, scale=.7]
    \tikzstyle{uStyle}=[shape = circle, minimum size = 6.0pt, inner sep = 0pt,
    outer sep = 0pt, draw, fill=white]
    \tikzstyle{lStyle}=[shape = rectangle, minimum size = 20.0pt, inner sep = 0pt,
outer sep = 2pt, draw=none, fill=none]
    \tikzset{every node/.style=uStyle}
    
    \begin{scope}[xshift=-2.7in]
    
    \draw (45:1.6cm) node(v1) {};
    \draw (135:1.6cm) node(v2) {};
    \draw (200:2.6cm) node(v3) {};
    \draw (340:2.6cm) node(v4) {};
    
    \foreach \i/\j in
    {1/2,2/3,3/4,4/1,1/3,2/4}
    \draw (v\i) edge (v\j);
    
    \draw (90:1.5cm) node[lStyle]{$i_{max}-i$};
    \draw (170:2cm) node[lStyle]{$i$};
    \draw (10:2cm) node[lStyle]{$i$};
    \draw (270:1.4cm) node[lStyle]{$\min\{i_{max}+i,m-i_{max}-i\}$};
    \draw (180:0.7cm) node[lStyle]{$i_{max}$};
    \draw (0:0.7cm) node[lStyle]{$i_{max}$};
    
    \end{scope}
    
    \begin{scope}[xshift=0in]
    
    \draw (45:1.6cm) node(v1) {};
    \draw (135:1.6cm) node(v2) {};
    \draw (200:2.6cm) node(v3) {};
    \draw (340:2.6cm) node(v4) {};
    
    \foreach \i/\j in
    {1/2,2/3,3/4,4/1,1/3,2/4}
    \draw (v\i) edge (v\j);
    
    \draw (90:1.5cm) node[lStyle]{$i_0$};
    \draw (170:2cm) node[lStyle]{$i$};
    \draw (10:2cm) node[lStyle]{$i$};
    \draw (270:1.4cm) node[lStyle]{$i_0+2i$};
    \draw (180:0.9cm) node[lStyle]{$i_{0}+i$};
    \draw (0:0.9cm) node[lStyle]{$i_{0}+i$};
    
    \end{scope}
    
    \begin{scope}[xshift=2.7in]
    
    \draw (45:1.6cm) node(v1) {};
    \draw (135:1.6cm) node(v2) {};
    \draw (200:2.6cm) node(v3) {};
    \draw (340:2.6cm) node(v4) {};
    
    \foreach \i/\j in
    {1/2,2/3,3/4,4/1,1/3,2/4}
    \draw (v\i) edge (v\j);
    
    \draw (90:1.5cm) node[lStyle]{$i_0-2i$};
    \draw (170:2cm) node[lStyle]{$i$};
    \draw (10:2cm) node[lStyle]{$i$};
    \draw (270:1.4cm) node[lStyle]{$i_0$};
    \draw (180:0.9cm) node[lStyle]{$i_{0}-i$};
    \draw (0:0.9cm) node[lStyle]{$i_{0}-i$};
    
    \end{scope}
    
    \end{tikzpicture}}
    \caption{Configurations of $S_i$, $S'_i$, and $S''_i$. Numbers represent the edge-lengths.}
    \label{fig:trapezoid}
\end{figure}

Now, by the inequality above and the assumption $y(\overline{S_i})\geq 0$, we conclude $(|I|-1)w_2-w_1< 0$, so $|I|< \frac{w_1+w_2}{w_2}$. By the pigeonhole principle, there exists $i_0\in I$ such that $y_{i_0}<\frac{w_2}{w_1+w_2}N$.

If $i_0\leq m'/2$, for each $1\leq i\leq m'/4$, we choose $S'_i\in \mathcal{I}(K_m,W)$ such that its $w_1$-weighted edges have lengths $i,i_0,i, i_0+2i$ following the cyclic order. See Figure~\ref{fig:trapezoid}. Then $\overline{S'_i}$ is uniquely determined and a similar analysis of the edge-lengths gives
\[\sum_{i=1}^{m'/4} y(\overline{S'_i})\leq \frac{m'}{4} w_1y_{i_0}+3(w_1+w_2)P<\frac{m'}{4}\frac{w_1w_2}{w_1+w_2}N+3(w_1+w_2)P<0,\]
where the last inequality is by the value of $m'$ and $N+P<0$. This contradicts the assumption that $y(\overline{S'_i})\geq 0$ for all possible $i$.

If $i_0> m'/2$, for each $1\leq i\leq m'/4$, we choose $S''_i\in \mathcal{I}(K_m,W)$ such that its $w_1$-weighted edges have lengths $i,i_0,i, i_0-2i$ following the cyclic order. See Figure~\ref{fig:trapezoid}. Then $\overline{S''_i}$ is uniquely determined and a similar analysis of the edge-lengths gives
\[\sum_{i=1}^{m'/4} y(\overline{S''_i})\leq \frac{m'}{4} w_1y_{i_0}+3(w_1+w_2)P<\frac{m'}{4}\frac{w_1w_2}{w_1+w_2}N+3(w_1+w_2)P<0,\]
where the last inequality is by the value of $m'$ and $N+P<0$. This is also a contradiction, hence we conclude the proof.
\end{proof}

\section{Proof of Theorem~\ref{higher}}\label{sec:higher}
The proof of Theorem~\ref{higher} is based on the following lemma.
\begin{lemma}\label{Kk_frac_pack}
For every integer $k>2$ there is an absolute constant $C_k$ with the following property. Suppose $W$ is a complete weighted cgg on $2k$ vertices with every length-$l$ edge having weight $w_l$. If $w_1>0$ and $w_{k}\geq C_k \sum_{l<k} w_l$, then there exists some $K_m$ having a perfect fractional $W$-packing.
\end{lemma}

\begin{proof}[Proof of Theorem~\ref{higher}] Let $C_k$ be as concluded by Lemma~\ref{Kk_frac_pack}. Recall that we wish to prove any cgg $G$ satisfying 
\[E_1\neq\emptyset\text{ and }|E_{k}|\geq C_k\cdot \sum_{l<k}|E_l|\]
is packable. Here $\{E_l; l=1,\dots,k\}$ are the edge collections associated with an interval partition $\mathcal{P}=\{V_1,\dots,V_{2k}\}$ as described in the statement of this theorem.

Denote the weighted representation of $G$ constructed from $\mathcal{P}$ as $W$ and its vertices as $v_1\prec \dots \prec v_{2k}$~($k>2$). Consider another weighted cgg $W'$ on these vertices such that every edge with length $l$ in $W'$ has weight $w_l$, where
\[w_{k}=2\times\sum_{l_W(e)=k}\omega_W(e)\text{\quad and\quad} w_l=\sum_{l_W(e)=l}\omega_W(e)\text{\quad for all $l=1,\dots, k-1$}.\]

By considering the rotations of $W$, we can check that $W'$ has a perfect fractional $W$-packing. Observe that we have $w_{k}=2|E_{k}|$ and $w_l=|E_l|$ for $1\leq l< k$. Then by hypothesis and Lemma~\ref{Kk_frac_pack}, some $K_m$ has a perfect fractional $W'$-packing, hence also a perfect fractional $W$-packing. So we conclude the proof by Observation~\ref{Km_packable} and Lemma~\ref{frac_pack_lemma}.\end{proof}

\begin{proof}[Proof of Lemma~\ref{Kk_frac_pack}]
Let $C_k=16k$ and $W$ be as described. Note that $k$ is the largest edge-length in $W$. We take
\[m=2m'+1\text{\quad with\quad}m'=\left\lceil 24k^3w_1^{-1}\sum_{l=1}^k w_l \right\rceil.\]
Note that $m'$ is the largest edge-length in $K_m$. For the rest of this proof, all the edge-lengths are defined in $K_m$.

By Observation~\ref{compressed_matrix}, it suffices to show there exists some vector $x\geq 0$ with $M_{K_m,W}x=\Vec{1}$ for $M_{K_m,W}$ defined in \eqref{eq:compressed_matrix_def}. For the sake of contradiction, suppose there's no such $x$, by Lemma~\ref{farkas}, there exists $y\in \mathbb{R}^{m'}$ such that\begin{itemize}
    \item $\sum^{m'}_{i=1}y_{i}=(\Vec{1})^T\cdot y<0$;
    \item For any $S\in \mathcal{I}(K_m,W)$, we have
    \[y(\overline{S}):=\sum_{e\in E(S)} \omega_S(e)\cdot y_{l_{K_m}(e)}=(M^T_{K_m,W}\cdot y)(\overline{S})\geq 0.\]
\end{itemize}

Let $I:=\{i\in [m']; y_i<0\}$, $N=\sum_{i\in I} y_i$, and $P=\sum_{i\not\in I} y_i$. Then we have $N+P<0$. Also, write $I'=\{i\in I; i<k\}$ and $N'=\sum_{i\in I'} y_i$.

For each $i\in I\setminus I'$, write $i=kq_i + r_i$ uniquely with integers $q_i>0$ and $0\leq r_i<k$. We choose $S_i\in \mathcal{I}(K_m,W)$ such that the lengths of its $w_1$-weighted edges listed following the cyclic order are
\[\underbrace{q_i,q_i,\dots,q_i}_{\text{$k-1$ times}}\ ,\ q_i+r_i\ ,\ \underbrace{q_i,q_i,\dots,q_i}_{\text{$k-1$ times}}\ ,\ \min\{2i-q_i-r_i,m-(2i-q_i-r_i)\}.\]
See Figure~\ref{fig:hex}. Note that $\overline{S_{i}}$ is uniquely determined in this way.

\begin{figure}[ht]
    \centering
    \begin{tikzpicture}[thick, scale=.7]
    \tikzstyle{uStyle}=[shape = circle, minimum size = 6.0pt, inner sep = 0pt,
    outer sep = 0pt, draw, fill=white]
    \tikzstyle{lStyle}=[shape = rectangle, minimum size = 20.0pt, inner sep = 0pt,
outer sep = 2pt, draw=none, fill=none]
    \tikzset{every node/.style=uStyle}
    
    \begin{scope}[xshift=-2.7in]
    
    \draw (45:1.6cm) node(v1) {};
    \draw (135:1.6cm) node(v2) {};
    \draw (190:2cm) node(v3) {};
    \draw (225:2.6cm) node(v4) {};
    \draw (315:2.6cm) node(v5) {};
    \draw (350:2cm) node(v6) {};
    \foreach \i/\j in
    {1/2,2/3,3/4,4/5,5/6,6/1}
    \draw (v\i) edge (v\j);
    
    \draw (90:1.5cm) node[lStyle]{$q_i+r_i$};
    \draw (16:1.9cm) node[lStyle]{$q_i$};
    \draw (164:1.9cm) node[lStyle]{$q_i$};
    \draw (207:2.5cm) node[lStyle]{$q_i$};
    \draw (333:2.5cm) node[lStyle]{$q_i$};
    \draw (270:2.4cm) node[lStyle]{$\min\{5q_i+r_i,m-5q_i-r_i\}$};
    
    \end{scope}
    
    \begin{scope}[xshift=0in]
    
    \draw (45:1.6cm) node(v1) {};
    \draw (135:1.6cm) node(v2) {};
    \draw (190:2cm) node(v3) {};
    \draw (225:2.6cm) node(v4) {};
    \draw (315:2.6cm) node(v5) {};
    \draw (350:2cm) node(v6) {};
    \foreach \i/\j in
    {1/2,2/3,3/4,4/5,5/6,6/1}
    \draw (v\i) edge (v\j);
    
    \draw (90:1.5cm) node[lStyle]{$i$};
    \draw (16:1.9cm) node[lStyle]{$j$};
    \draw (164:1.9cm) node[lStyle]{$j$};
    \draw (207:2.5cm) node[lStyle]{$j$};
    \draw (333:2.5cm) node[lStyle]{$j$};
    \draw (270:2.4cm) node[lStyle]{$i+4j$};
    
    \end{scope}
    
    \end{tikzpicture}
    \caption{Configurations of $S_i$ and $S_{ij}$ for $k=3$. Numbers represent the edge-lengths.}
    \label{fig:hex}
\end{figure}
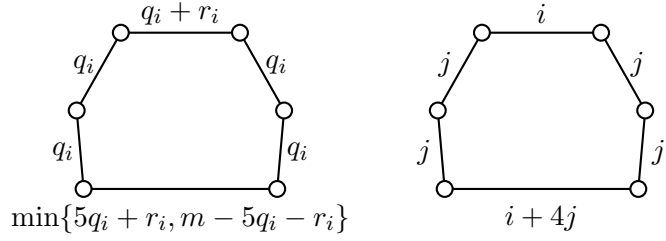

Observe that all $w_k$-weighted edges of $S_i$ have length $i$ and there are $k$ many of them. On the other hand, for any $l\in [k]$, every $w_l$-weighted edge of $S_i$ has length equal to either $lq_i$, or $lq_i+r_i$, or $i+(k-l)q_i$, or $m-(i+(k-l)q_i)$. And for each such value, there are at most $2k$ many such edges of $S_i$ having it as their length. Hence we can check, among all chosen $S_i$, there are at most $8k^2$ many $w_{l}$-weighted edges having length $j$ for each $j\in [m']\setminus I$. Overall we have
\begin{align*}
    \sum_{i\in I\setminus I'} y(\overline{S_i})&\leq kw_{k}(N-N')+ 8k^2\sum_{l< k} w_lP\\
    &=kw_{k}\left(\frac{1}{2}N-N'+\frac{1}{2}N\right)+ 8k^2\sum_{l< k} w_lP\\
    &\leq kw_{k}\left(\frac{1}{2}N-N'\right) + \frac{1}{2}kC_k\sum_{l< k} w_lN+8k^2\sum_{l< k} w_lP\\
    &< kw_{k}\left(\frac{1}{2}N-N'\right),
\end{align*}where the last inequality is by the value of $C_k$ and $N+P<0$.

Now, by the inequality above and the assumption $y(\overline{S_i})\geq 0$, we conclude $N'<N/2$. For each $i\in I'$ and $j\leq \frac{m'}{2k}$, we choose $S_{ij}\in \mathcal{I}(K_m,W)$ such that the lengths of its $w_1$-weighted edges listed following the cyclic order are
\[\underbrace{j,j,\dots,j}_{\text{$k-1$ times}}\ ,\ i\ ,\ \underbrace{j,j,\dots,j}_{\text{$k-1$ times}}\ ,\ i+(2k-2)j.\]
See Figure~\ref{fig:hex}. Note that $i+(2k-2)j\leq m'$ in this definition. Again, $\overline{S_{ij}}$ is uniquely determined and a similar analysis of the edge-lengths gives
\[\sum_{j=1}^{m'/2k}y(\overline{S_{ij}})\leq \frac{m'}{2k}w_1 y_i + 6k\sum_{l} w_lP.\]
Hence, as $|I'|<k$ and $N'<N/2$, we have
\[\sum_{i\in I'}\sum_{j=1}^{m'/2k}y(\overline{S_{ij}}) < \frac{m'}{2k}w_1 N'+ 6k^2\sum_{l} w_lP< \frac{m'}{2k}w_1 \frac{N}{2}+ 6k^2\sum_{l} w_lP<0,\]
where the last inequality is by the value of $m'$ and $N+P<0$. This contradicts the assumption that $y(\overline{S_{ij}})\geq 0$ for all $S_{ij}$, so we conclude the proof.
\end{proof}

\section{Proof of Theorem~\ref{ordered}}\label{sec:ordered}
Some details of this proof are abridged due to their similarity to previous arguments.
\begin{proof}[Proof of Theorem~\ref{ordered}]
Let $G$ be the given ordered graph. If $\chi_<(G)=2$, a result of Pach and Tardos~\cite{pach2006forbidden} states that
\[\text{ex}_{<}(n,G)=\left(1-\frac{1}{\chi_<(G)-1}\right)\binom{n}{2}+o(n^2)=o(n^2),\]
where $\text{ex}_{<}(n,G)$ is the maximum number of edges in a $n$-vertex ordered graph that doesn't contain $G$ as an ordered subgraph. Similar to the $\chi_c(G)=2$ case in Theorem~\ref{main}, we conclude $G$ is packable.

If $\chi_<(G)=3$, let $\{V_1,V_2,V_3\}$ be an interval partition of $V(G)$ with every edge in $E(G)$ having endpoints in different parts. Write $e_{ij}=|E_G(V_i,V_j)|$ for $1\leq i<j\leq 3$. For each positive integer $n$, the \emph{irregular blow-up} $\tilde{K}_3[n]$ is an ordered graph defined as follows: the vertex set $V(\tilde{K}_3[n])$ consists of three intervals $I_1,I_2,I_3$ in order, of size $e_{12}e_{13}n$, $e_{12}e_{23}n$, $e_{13}e_{23}n$ respectively; the edge set $E(\tilde{K}_3[n])$ consists of all pairs with endpoints in different parts.

First we prove that given arbitrary $\epsilon>0$, there exists $n_\epsilon$ such that, for all $n\geq n_\epsilon$ there is a $G$-packing of $\tilde{K}_3[n]$ covering all but an $\epsilon$-fraction of edges. Consider the hypergraph $\mathcal{H}$ whose vertex set $\mathcal{V}$ consists of all edges of $\tilde{K}_3[n]$ and edge set $\mathcal{E}$ consists of all $E(G')$ where $G'$ is a subgraph of $\tilde{K}_3[n]$ satisfying\begin{itemize}
    \item there exists an isomorphism $f:G\to G'$ such that with $f(V_i)\subset I_i$ for $i=1,2,3$; and
    \item there exists $\delta\in [1,\log n]$ such that for $i=1,2,3$, every consecutive pair of vertices in $V(G')\cap I_i$ has length $\delta$ in $\tilde{K}_3[n]$,
\end{itemize} Using similar arguments as in the proof of Lemma~\ref{frac_pack_lemma}, we can check inside $\mathcal{H}$: All $x\in \mathcal{V}$ but $O(n\log n)$ of them have degree $d(x)=(1\pm o(1))e_{12}e_{13}e_{23}n\log n$; All $x\in \mathcal{V}$ have degree $d(x)<2e_{12}e_{13}e_{23}n\log n$; Any two distinct $x_1,x_2\in \mathcal{V}$ have codegree $d(x,y)<O(e_{12}e_{13}e_{23}n)$. Therefore, by Theorem~\ref{Pippenger_Spencer}, there exists $n_\epsilon$ with the claimed property.

Now we fix an arbitrary $\epsilon>0$ and describe a recursive construction of a $G$-packing of $K_n$. Firstly, we partition $V(K_n)$ into four intervals $I_1,I_2,I_3,I_4$ in order, where $I_1,I_2,I_3$ have sizes $e_{12}e_{13}n'$, $e_{12}e_{23}n'$, $e_{13}e_{23}n'$ respectively, for a unique $n'$ such that $I_4$ has size less than $q:=e_{12}e_{13}+e_{12}e_{23}+e_{13}e_{23}$. Then we can regard the subgraph of $K_n$ with vertices $I_1\cup I_2\cup I_3$ and edges $E_{K_n}(I_1,I_2)\cup E_{K_n}(I_1,I_3)\cup E_{K_n}(I_2,I_3)$ as $\tilde{K}_3[n']$. If $n'< n_\epsilon$, we do nothing and end this construction. Otherwise, we find a $G$-packing of this $\tilde{K}_3[n']$ that covers all but an $\epsilon$-fraction of edges, and we recursively repeat this process on the induced subgraphs on $I_1,I_2,I_3$ respectively. Finally, collecting all the produced edge-disjoint $G$-copies gives the resulting $G$-packing of $K_n$.

For $n$ sufficiently large, the uncovered edges of above $G$-packing are in either one of the following categories:\begin{itemize}
    \item Edges with a endpoint in $I_4$ at some stage. There are $O(n\log n)$ many of them;
    \item Edges in some $\tilde{K}_3[n']$ with $n'< n_\epsilon$. There are at most $\frac{n}{qn_\epsilon}\binom{qn_\epsilon}{2}=O(n)$ many of them.
    \item Edges in some $\tilde{K}_3[n']$ with $n'\geq n_\epsilon$ but not covered by the constructed $G$-packing. There are at most $\epsilon \binom{n}{2}$ many of them.
\end{itemize} Therefore, there exists a $G$-packing of $K_n$ covering all but $\epsilon n^2$ edges. By the arbitrariness of $\epsilon$, we conclude that $G$ is packable.
\end{proof}

\section{Final remarks}\label{sec:remark}
\noindent 1. There is a polynomial time randomized greedy algorithm for the concluded matching in Theorem~\ref{Pippenger_Spencer} \cite{rodl1996asymptotic}. It is also well-known that linear programs are solvable in polynomial time \cite[Chapter 7]{matousek2006understanding}. Hence there exist polynomial time randomized algorithms for the packings asserted by our theorems.

\medskip

\noindent 2. We believe our methods could also prove Theorem~\ref{higher} where the interval partition $\mathcal{P}$ has an odd size. But such a proof probably requires new constructions like $S_i\in \mathcal{I}(K_m,W)$ in Lemma~\ref{Kk_frac_pack}.

\medskip

\noindent 3. Is it true that for any cgg $G$ with $\chi_c(G)=k\geq 5$, there exist $G$-packings of $K_n$ covering a $(1-o(1))4/k$-fraction of edges? If this is possible, it is going to be asymptotically tight, by an average length argument.

\medskip\noindent {\bf Acknowledgement.} We want to thank Daniel Cranston, Andrew Suk, Jacques Verstra\"ete, Alexandra Wesolek, and Yunkun Zhou for helpful discussions. We also wish to thank every organizers of Graduate Student Combinatorics Conference 2022 where this work is initiated.

\bibliographystyle{abbrv}
\bibliography{bib}
\end{document}